\documentclass[11pt,a4paper,headinclude,footinclude]{scrartcl}                 
\usepackage[T1]{fontenc}                   
\usepackage[utf8]{inputenc}                 
\usepackage[english]{babel}       
\usepackage{graphicx}                       %
\usepackage[font=small]{quoting}            %
\usepackage{mparhack,fixltx2e,relsize}      %
\usepackage{comment}      
\usepackage{caption}                  %
\usepackage[nochapters,beramono,eulermath,%
            pdfspacing,
            listings,
            ]{classicthesis}                %
\usepackage{arsclassica}                    
\usepackage[english]{babel}
\usepackage{verbatim}
\usepackage{color}
\usepackage{picture}
\usepackage{amsmath,amsthm,amsfonts,amssymb}
\usepackage{comment}
\usepackage{titling} 
\usepackage[hyperpageref]{backref}

\newtheorem{theo}{Theorem}[section]

\newtheorem{cor}{Corollary}[section]

\theoremstyle{definition}
\newtheorem{definiz}{Definition}[section]
\newtheorem{claim}{Claim}
\newtheorem{rem}{Remark}[section]
\newtheorem{ex}{Example}[section]
\numberwithin{equation}{section}

\newcommand{\R}{\mathbb R}

\newcommand{\de}{\partial}
\newcommand{\eps}{\varepsilon}

\newcommand{\lto}{\left(}
\newcommand{\rto}{\right)}

\DeclareMathOperator{\divergenza}{div}

\DeclareMathOperator{\sign}{sign}

\newtoks\by
\newtoks\paper
\newtoks\book
\newtoks\jour
\newtoks\yr
\newtoks\pg
\newtoks\vol
\newtoks\publ
\def\ota{{\hbox\vol{???}}}
\def\cLear{\by=\ota\paper=\ota\book=\ota\jour=\ota\yr=\ota
\pg=\ota\vol=\ota\publ=\ota}
\def\endpaper{\the\by, \textit{\the\paper},
{\the\jour} \textbf{\the\vol} (\the\yr), \the\pg.\cLear}
\def\endbook{\the\by, \textit{\the\book}, \the\publ.\cLear}
\def\endprep{\the\by, \textit{\the\paper}, \the\jour.\cLear}
\def\endproc{\the\by, \textit{\the\paper}, \the\book, \the\publ,
\the\yr, \the\pg.\cLear}

\begin{document}

\title{Anisotropic Hardy inequalities}
\author{Francesco Della Pietra%
\thanks{Universit\`a degli studi di Napoli Federico II, Dipartimento di Matematica e Applicazioni ``R. Caccioppoli'', Via Cintia, Monte S. Angelo - 80126 Napoli, Italia. Email: f.dellapietra@unina.it} 
{, }\\
Giuseppina di Blasio%
\thanks{Seconda Universit\`a degli studi di Napoli, Dipartimento di Ma\-te\-ma\-ti\-ca e Fisica, Via Vivaldi, 43 - 81100 Caserta, Italia. Email: giuseppina.diblasio@unina2.it},\\
Nunzia Gavitone
\thanks{Universit\`a degli studi di Napoli Federico II, Dipartimento di Matematica e Applicazioni ``R. Caccioppoli'', Via Cintia, Monte S. Angelo - 80126 Napoli, Italia. Email: nunzia.gavitone@unina.it} 
}
\date{}
\pagestyle{scrheadings} 
\maketitle
\begin{abstract}
{\sc Abstract.} We study some Hardy-type inequalities involving a general norm in 
$\R^{n}$ and an anisotropic distance function to the boundary. 
The case of the optimality of the constants is also addressed.\\

\end{abstract}

\section{Introduction}
Let $F$ be a smooth norm of $\R^{n}$. In this paper we investigate the validity of Hardy-type inequalities
\begin{equation}
 \label{introth}
 \displaystyle \int_{\Omega} F(\nabla  u)^2 \,dx \ge 
C_{F}(\Omega) \displaystyle 
\int_{\Omega} \frac{u^2}{d_F^2} \,dx, 
\qquad \forall u \in H_0^1(\Omega),
\end{equation}
where $\Omega$ is a domain of $\R^{n}$, and $d_{F}$ is the anisotropic distance to the boundary with respect the dual norm (see Section 2 for the precise assumptions and definitions). 
We aim to study the best possible constant for which \eqref{introth} holds, in the sense that
\[
C_{F}(\Omega)=\inf_{u\in H_{0}^{1}(\Omega)} \frac{\displaystyle\int_{\Omega} F(\nabla  u)^2 \,dx}{\displaystyle\int_{\Omega}\frac{u^2}{d_{F}^2} \,dx}.
\]

In the case of the Euclidean norm, that is when $F=\mathcal E=|\cdot |$, the inequality \eqref{introth} reduces to
\begin{equation}
 \label{hardy:intro} \displaystyle \int_{\Omega} |\nabla  u|^2 \,dx \ge 
C_{\mathcal E}(\Omega)\displaystyle \int_{\Omega} \frac{u^2}{d_{\mathcal E}^2} \,dx,
\qquad  \forall u \in H_0^1(\Omega),
\end{equation}
where 
\[
d_{\mathcal E}(x)=\inf_{y\in\de\Omega}|x-y|,\qquad x\in\Omega
\] 
is the usual distance function from the boundary of $\Omega$, and $C_{\mathcal E}$ is the best possible constant.

The inequality \eqref{hardy:intro} has been studied by many authors, under several points of view. For example, it is known that for any bounded domain with Lipschitz boundary $\Omega$  of $\R^{n}$, $0<C_{\mathcal E}(\Omega)\le \frac 14$ (see \cite{da95,mmp,bm97}). In particular, if $\Omega$ is a convex domain of $\R^{n}$, the optimal constant $C_{\mathcal E}$ in \eqref{hardy:intro} is independent of $\Omega$, and its value is $C_{\mathcal E}=\frac 14$, but there are smooth bounded domains such that $C_{\mathcal E}(\Omega)<\frac 14$ (see \cite{mmp,ms97}). Furthermore, in \cite{mmp} it is proved that $C_{\mathcal E}$ is achieved if and only if it is strictly smaller than $\frac14$.

Actually, the value of the best constant $C_{\mathcal E}(\Omega)$ is still $\frac 14$ for a more general class of domains. This has been shown, for example, in \cite{bft}, under the assumption that $d_{\mathcal E}$ is superharmonic in $\Omega$, in the sense that 
\begin{equation}
\label{superintro}
\Delta d_{\mathcal E} \le 0\quad\text{in }\mathcal D'(\Omega).
\end{equation}
As proved in \cite{lll}, when $\Omega$ is a $C^{2}$ domain the condition \eqref{superintro} is equivalent to require that $\de\Omega$ is weakly mean convex, that is its mean curvature is nonnegative at any point. 

The fact that the constant $C_{\mathcal E}(\Omega)=\frac 14$ is not achieved has lead to the interest of studying  ``improved'' versions of \eqref{hardy:intro}, by adding remainder terms which depend, in general, on suitable norms of $u$. 
For instance, if $\Omega$ is a bounded and convex set, in \cite{bm97} it has been showed that
\begin{equation}
\label{bm}
\displaystyle \int_{\Omega} |\nabla  u|^2 \,dx \ge 
\frac 14\displaystyle \int_{\Omega} \frac{u^2}{d_{\mathcal E}^2} \,dx + \frac{1}{4L^{2}}\displaystyle \int_{\Omega} u^2\,dx, \qquad  \forall u \in H_0^1(\Omega),
\end{equation}
where $L$ is the diameter of $\Omega$. Actually, in \cite{hhl} the authors showed that the value $ \frac{1}{4L^{2}}$ can be replaced by a constant which depends on the volume of $\Omega$, namely $ {c(n)}{|\Omega|^{{-\frac 2n}}}$; here $c(n)$ is suitable constant depending only on the dimension of the space (see also \cite{fmt06}).

When $\Omega$ satisfies condition \eqref{superintro}, several improved versions of \eqref{hardy:intro} can be found for instance in \cite{bft03,bft,fmt06}. More precisely, in \cite{bft} the authors proved that 
\begin{equation}
\label{imprintro}
	\int_{\Omega}|\nabla u|^{2} dx - \frac{1}{4}
	\int_{\Omega}\frac{|u|^{2}}{d_{{\mathcal{E}}}^{2}}dx \ge  \frac{1}{4} 
	 \int_{\Omega} \frac{ |u|^{2} }{ d_{\mathcal{E}}^{2}}
\left(\log\frac{D_{0}}{d_{\mathcal{E}}}\right)^{-2}	 
	 dx,
\end{equation}
where  $D_{0}\ge e\cdot\sup\{ d_{\mathcal{E}}(x,\de \Omega)\}$  and  $u\in H_0^1(\Omega)$. 

The aim of this paper is to study Hardy inequalities of the type \eqref{introth}, and to show improved versions in the anisotropic setting given by means of the norm $F$, in the spirit of \eqref{imprintro}. For example, one of our main result states that for suitable domains $\Omega$ of $\R^{n}$, and for every function $u \in H_0^1(\Omega)$, it holds that
\begin{equation}
 \label{mainintro}
\int_{\Omega} F^{2}(\nabla  u) \,dx - \frac{1}{4}\displaystyle \int_{\Omega} \frac{u^2}{d_F^2} \,dx \ge\frac14\int_{\Omega} \frac{u^{2}}{d_{F}^{2}}\left(\log\frac{D}{d_F}\right)^{-2}
dx,
 \end{equation}
where $D= e\cdot\sup\{d_{F}(x,\de \Omega),x\in\Omega\}$.

The condition we will impose on $\Omega$ in order to have \eqref{mainintro} will involve the sign of an anisotropic Laplacian of $d_{F}$ (see sections 2 and 3). We will also show that such condition is, in general, not equivalent to \eqref{superintro}. 

Actually, we will deal also with the optimality of the involved constants.
Moreover, we will show that \eqref{mainintro} implies an improved version of \eqref{introth} in terms of the $L^{2}$ norm of $u$, in the spirit of \eqref{bm}. More precisely, we will show that if $\Omega$ is a convex bounded open set then
 \begin{equation*}
 \int_{\Omega} F(\nabla u)^{2} dx -  \frac{1}{4} \int_{\Omega} \frac{ |u|^{2} }{ d_F^{2}}dx \ge C(n)|\Omega|^{-\frac{2}{n}} \int_{\Omega}|u|^{2} \,dx.
\end{equation*}

We emphasize that Hardy type inequalities in anisotropic settings have been studied, for example, in \cite{schaft,zaa,bfra}, where, instead of considering the weight $d_{F}^{-2}$, it is taken into account a function of the distance from a point of the domain (see for example \cite{bft,gaper,bv97,hlp,vz,afv,bct} and the reference therein for the Euclidean case).

The structure of the paper is the following. In Section 2 we fix the necessary notation and provide some preliminary results which will be needed later. Moreover, we discuss in some details the condition we impose on $\Omega$ in order to be \eqref{mainintro} and \eqref{introth} true. In Section 3 we study the inequality \eqref{introth} and we give some applications. In Section 4 the improved versions of \eqref{introth} are investigated and, finally, Section 5 is devoted to the study of the optimality of the constants in \eqref{mainintro}.

\section{Notation and preliminaries}
Throughout the paper we will consider a convex even 1-homogeneous function 
\[
\xi\in \R^{n}\mapsto F(\xi)\in [0,+\infty[,
\] 
that is a convex function such that
\begin{equation}
\label{1hom}
F(t\xi)=|t|F(\xi), \quad t\in \R,\,\xi \in \R^{n}, 
\end{equation}
 and such that
\begin{equation}
\label{crescita}
\alpha_{1}|\xi| \le F(\xi),\quad \xi \in \R^{n},
\end{equation}
for some constant $0<\alpha_{1}$. 
Under this hypothesis it is easy to see that there exists $\alpha_{2}\ge \alpha_{1}$ such that
\[
H(\xi)\le \alpha_{2} |\xi|,\quad \xi \in \R^{n}.
\]
Furthermore we suppose that $F^{2}$ is strongly convex, in the sense that $F\in C^{2}(\R^{n}\setminus \{0\})$ and
\[
\nabla_{\xi}^{2} F^{2}>0 \qquad \text{in }\R^{n}\setminus\{0\}.
\]
%
%
%

In this context, an important role is played by the polar function of $F$, namely the function $F^{o}$ defined as
\[
x\in \R^{n}\mapsto F^{o}(x)= \sup_{\xi \not=0}\frac{\xi\cdot x}{F(\xi)}.
\]
It is not difficult to verify that $F^{o}$ is a convex, $1$-homogeneous function that satisfies
\begin{equation}
\label{crescitazero}
\frac{1}{\alpha_{2}} |\xi| \le F^{o}(\xi) \le \frac{1}{\alpha_{1}} |\xi|,\quad \forall \xi\in\R^{n}. 
\end{equation}
Moreover, the hypotheses on $F$ ensures that $F^o \in C^2(\R^n\setminus \{0\})$ (see for instance \cite{schn})
\[
F(x)=\lto F^{o}\rto^{o}(x)= \sup_{\xi \not=0}\frac{\xi\cdot x}{F^{o}(\xi)}.
\]

The following well-known properties hold true:
\begin{align}
 &F_\xi(\xi)\cdot \xi = F(\xi),\quad \xi\ne 0, \label{eq:eul} \\
 &F_\xi ( t \xi ) = \sign t \cdot F_\xi(\xi), \quad \xi \ne 0, \, t \ne 0, \label{0-om}\\
 &\nabla^{2}_{\xi}F(t\xi)=\frac{1}{|t|}\nabla^{2}_{\xi}F(\xi)\quad\xi \ne 0, \, t \ne 0,\\
 & F\lto F_\xi^o(\xi)\rto=1,\quad \forall \xi \not=0, \label{eq:F1} \\
 &F^o(\xi)\, F_{\xi}\lto F^o_{\xi}(\xi) \rto = \xi \quad \forall \xi \not=0.  \label{eq:FF0}
\end{align} 
Analogous properties hold interchanging the roles of $F$ and $F^{o}$. 

The open set
\[
\mathcal W = \{  \xi \in \R^n \colon F^o(\xi)< 1 \}
\]
is the so-called Wulff shape centered at the origin.
More generally, we denote 
\[
\mathcal W_r(x_0)=r\mathcal W+x_0=\{x\in \R^{2}\colon F^o(x-x_0)<r\},
\] 
and $\mathcal W_r(0)=\mathcal W_r$.

We recall the definition and some properties of anisotropic curvature for a smooth set. For further details we refer the reader, for example, to \cite{atw} and \cite{bp}.
\begin{definiz}
\label{normale}
Let $A \subset \R^n$ be a open set with smooth boundary.  The anisotropic outer normal $n_{A}$ is defined as
  \[
  n_A(x)= \nabla_{\xi} F(\nu_A(x)),\quad x\in \partial A,
  \]
  where $\nu_A$ is the unit Euclidean outer normal to $\de A.$
  \end{definiz}
  \begin{rem}
  \label{normwulff}
We stress that if $A=\mathcal W_{r}(x_{0})$, by the properties of $F$ it follows that
\[
n_A(x)= \nabla_\xi F\big(\nabla_{\xi} F^o(x-x_0) \big) =
\frac{1}{r}(x-x_0),\quad x\in \de A.
\]
\end{rem}

Finally, let us recall the definition of Finsler Laplacian
\begin{equation}
\label{lapani}
\Delta_F  u= \text{div}(F(\nabla u)F_\xi(\nabla u)), 
\end{equation}
defined for $u\in H^1_0(\Omega).$ 

\begin{definiz} \label{super}
Let $u\in H^1_0(\Omega)$, we say that $u$ is $\Delta_F$-superharmonic if 
\begin{equation}
\label{C}
\tag{$C_{H}$}
-\Delta_{F}u \ge 0 \quad\text{in }\mathcal D',
\end{equation}
that is
\[
\int_{\Omega}F(\nabla u)F_{\xi}(\nabla u) \cdot \nabla \varphi\, dx \ge 0, \qquad \forall \varphi \in C_{0}^{\infty}(\Omega),\,\varphi \ge 0.
\]
\end{definiz}
Similarly, by writing that $u$ is $\Delta$-superharmonic we will mean that $u$ is superharmonic in the usual sense, that is $-\Delta u \ge 0$.

\subsection{\bf{Anisotropic distance function}}
Due to the nature of the problem, it seems to be natural to consider a suitable notion of distance to the boundary. 

Let us consider a domain $\Omega$, that is a connected open set of $\R^n$, with non-empty boundary. 

The anisotropic distance of $x\in\overline\Omega$ to the boundary of $\de \Omega$ is the function 
\begin{equation}
\label{defdist}
d_{F}(x)= \inf_{y\in \de \Omega} F^o(x-y), \quad x\in \overline\Omega.
\end{equation}

We stress that when $F=|\cdot|$ then $d_F=d_{\mathcal{E}}$, the Euclidean distance function from the boundary.

It is not difficult to prove that $d_{F}$ is a uniform Lipschitz function in $\overline \Omega$ and, using the property \eqref{eq:F1},
\begin{equation}
  \label{Fd}
  F(\nabla d_F(x))=1 \quad\text{a.e. in }\Omega.
\end{equation}
Obviously, assuming $\sup_{\Omega} d_{F}<+\infty$, $d_F\in W_{0}^{1,\infty}(\Omega)$ and the quantity
\begin{equation}
\label{inrad}
r_{F}=\sup \{d_{F}(x),\; x\in\Omega\},
\end{equation}
is called  anisotropic inradius of $\Omega$.

For further properties of the anisotropic distance function we refer
the reader to \cite{cm07}.

The main assumption in this paper will be that $d_F$ is $\Delta_F$-su\-per\-harmonic in the sense of Definition \ref{super}.

\begin{rem}
We emphasize that if $\Omega$ is a convex set, the functions $d_{\mathcal{E}}$ and $d_{F}$ are respectively $\Delta$ and $\Delta_F$-superharmonic. This can be easily proved by using the concavity of $d_{\mathcal{E}}$ and $d_{F}$ in $\Omega$ (see for instance \cite{eg}, and \cite{pota} for the anisotropic case). 

Actually, in the Euclidean case there exist non-convex sets for which $d_{\mathcal{E}}$  is still $\Delta$-superharmonic. An example can be obtained for instance in dimension $n=3$, taking the standard torus (see \cite{ak}). Similarly, in the following example we show that there exists a non-convex set such that the anisotropic distance function $d_F$ is $\Delta_F$-superharmonic.
\end{rem}

\begin{ex}\label{ete}
Let us consider the following Finsler norm in $\R^3$
\[
F(x_1,x_2,x_3)=(x^2_1+x^2_2+a^2x^2_3)^\frac{1}{2},
\]
with $a>0$; then 
\[
F^o(x_1,x_2,x_3)=\left(x^2_1+x^2_2+\frac{x^2_3}{a^2}\right)^\frac{1}{2}.
\]

We consider the set $\Omega\subset \R^3$ obtained by rotating the ellipse  
\[
\gamma=\{(0,x_2,x_3):(x_2-R)^2+\frac{x^2_3}{a^2}<r^2\} \quad \text{with} \quad R>r>0,
\]
about the $x_3-$axis. Obviously $\Omega $ is not convex. In order to show that \eqref{C} holds, we first observe that if we fix a generic point $x\in \gamma$, then being  $F$ isotropic with respect to the first two components the anisotropic distance is achieved in a point $\overline x$ of the boundary of $\gamma$. Moreover it is not difficult to show that the vector $\overline x-x$ has the same direction of the anisotropic normal $n_\Omega$ (see Definition \ref{normale}, and see also \cite{dpg1}).  Hence by Remark \ref{normwulff}
\[
d_F(x)=r-F^o(x-x_0),
\]
where $x_0=(0,R,0)$ is the center of the ellipse. 

Now let us introduce a plane polar coordinates $(\rho, \vartheta)$ such that to a generic point $Q=(x_1,x_2,x_3)\in \R^3$, it is associated the point $Q'=(\rho \cos \vartheta, \rho \sin \vartheta,x_3)$, where $\rho=\sqrt{x^2_1+x^2_2}$ and $\vartheta\in [0,2\pi]$. Then, by construction, 
\[
\Omega=\{Q'\in \R^3: F^o(Q'-C)<r\},
\]
    where $C=(R \cos\vartheta, R \sin\vartheta, 0)$ and $F^o(Q'-C)^2=\left(R-\rho \right)^2+\frac{x^2_3}{a^2}$ . 

Then as observed before, fixed $Q'\in \Omega$
\begin{multline*}
d_F(Q')=r-F^o(Q'-C)=\\ =r-\sqrt{\left(R-\rho \right)^2+\frac{x^2_3}{a^2}}=r-\sqrt{\left(R-\sqrt{x^2_1+x^2_2} \right)^2+\frac{x^2_3}{a^2}}.
\end{multline*}

Now we are in position to prove \eqref{C}. We note that for all $Q'\neq C$
\begin{gather}
\begin{split}
\Delta_F  d_F(Q') &=\text{div}(F(\nabla d_F)F_\xi(\nabla d_F))\\
&=\frac{\de^2 d_F}{\de x^2_1}+\frac{\de^2 d_F}{\de x^2_2} +a^2 \frac{\de^2 d_F}{\de x^2_3}\\
&=\frac{1}{\rho}\frac{\de d_F }{\de \rho}+\frac{\de^2 d_F }{\de \rho^2}+a^2 \frac{\de^2 d_F}{\de x^2_3}\\
&=\frac{R-2\rho}{\rho F^o(Q'-C)}.
\end{split}
\end{gather}
Being $\varrho>R-r$, we get that  $d_F $ is $\Delta_F$-superharmonic in $\Omega$ if $R>2r$ for all $a>0$.
\end{ex}

\begin{rem}
\label{equiv}
In general,  if $\Omega$ is not convex, to require that $d_{F}$ is $\Delta_F$-superharmonic does not assure that  $d_{\mathcal E}$ is  $\Delta$-superharmonic. Indeed, let $\Omega$ be as in Example \ref{ete}; if we take $R\ge 2r$ then, as showed before, $-\Delta_F  d_F\geq0$. On the other hand it is possible to choose $a>0$ such that $d_{\mathcal E}$ is not $\Delta$-superharmonic. To do that, it is enough to prove that the mean curvature of $\Omega$ is negative at some points of the boundary for a suitable choice of $a$. Indeed in \cite{lll} it is proved that $d_{\mathcal E}$ is $\Delta$-superharmonic on $\Omega$ if and only if the mean curvature $H_\Omega (y)\ge 0$ for all $y\in \de\Omega$.

The parametric equations of $\de \Omega$ are 
\begin{center}
$
\varphi(t,\vartheta):
\begin{cases}
x_1=(R+r\cos \vartheta) \cos t= \phi(\vartheta) \cos t\\
x_2=(R+r\cos \vartheta) \sin t=\phi(\vartheta) \sin t\\
x_3=a r \sin \vartheta= \psi(\vartheta),
\end{cases}$
\end{center}
where $t,\vartheta \in [0,2 \pi]$.

Then for $y=\varphi(t,\vartheta) \in \de \Omega$ we have

\begin{gather}
\begin{split}
H_\Omega(y)&= -\displaystyle \frac{\phi (\phi'' \psi'-\phi'\psi'')-\psi'((\phi')^2+(\psi')^2)}{2 |\phi|((\phi')^2+(\psi')^2)^{\frac{3}{2}}}\\ 
 &= \displaystyle \frac{a r^2 (R+2r\cos\vartheta +r\cos^3\vartheta (a^2-1))}{2|R+r \cos \vartheta | (r^2 \sin^2 \vartheta+a^2r^2 \cos^2\vartheta)^{\frac{3}{2}}}.
\end{split}
\end{gather}
Finally we observe that if $\vartheta=\pi$ then $H_{\Omega}(y)<0,$  if $a >1$.  
\end{rem}

\section{Anisotropic Hardy inequality}
\begin{theo}
\label{hardyclass}
Let $\Omega$ be a domain in $\R^n$ and suppose that condition \eqref{C} holds. Then for every function $u \in H_0^1(\Omega)$ the following anisotropic Hardy inequality holds
\begin{equation}
 \label{hardy}
 \displaystyle \int_{\Omega} F(\nabla  u)^2 \,dx \ge 
\frac{1}{4}\displaystyle \int_{\Omega} \frac{u^2}{d_F^2} \,dx
\end{equation}
where $d_F$ is the anisotropic distance function from the boundary of $\Omega$ defined in \eqref{defdist}.
\end{theo}
\begin{proof}
First we prove inequality \eqref{hardy}.
Being $F^{2}$ convex, we have that
\[
F(\xi_{1})^{2}\ge F(\xi_{2})^{2}+2F(\xi_{2})F_{\xi}(\xi_{2})\cdot (\xi_{1}-\xi_{2}).
\]
Hence putting $ \xi_{1}= \nabla u$ and $\xi_{2}=Au\frac{\nabla d_{F}}{d_{F}}$,  with $A$ positive constant, recalling that $F(\nabla d_{F})=1$, by the  homogeneity of $F$ we get
\begin{gather*}
   \displaystyle\int_{\Omega} F(\nabla u)^2\,dx\ge -A^2\displaystyle\int_{\Omega} \frac{u^2}{d_F^2}  \,dx+A\displaystyle\int_{\Omega} \frac{2u}{d_F} F_{\xi}(\nabla d_{F}) \cdot \nabla u \, dx  . 
\end{gather*}   
By the Divergence Theorem (in a general setting, contained for example in \cite{anz}) we have
\begin{gather*}
 \displaystyle\int_{\Omega} F(\nabla u)^2\,dx\ge-A^2\displaystyle\int_{\Omega} \frac{u^2}{d_F^2}  \,dx+A\displaystyle\int_{\Omega} \frac{F_{\xi}(\nabla d_{F})}{d_F}  \cdot \nabla (u^{2}) \, dx  \\
 \ge -A^2\displaystyle\int_{\Omega} \frac{u^2}{d_F^2}  \,dx-A\displaystyle\int_{\Omega}  u^{2 }\frac{\Delta_{F} d_{F}}{d_F}  \, dx+A  \displaystyle\int_{\Omega} \frac{u^2}{d_F^2}  \,dx
\end{gather*}

Being $-\Delta_{F} d_F \ge 0$ we get
\begin{gather*}
  \displaystyle\int_{\Omega} F(\nabla u)^2\,dx \ge (A-A^2)\displaystyle \int_{\Omega} \frac{u^2}{d_F^2} \,dx. 
\end{gather*}
Then maximizing with respect to $A$ we obtain that $A=\frac 1 2$, and then \eqref{hardy} follows.

\end{proof}
\begin{rem}
\label{alphabeta}
We observe that if $\Omega$ is a convex domain in $\R^{n}$, an  inequality of the type \eqref{hardy} can be immediately obtained by using the following classical  Hardy inequality involving the Euclidean distance function  $d_{\mathcal{E}}$
\begin{equation}
 \label{hardyeu} \displaystyle \int_{\Omega} |\nabla  u|^2 \,dx \ge 
\frac{1}{4}\displaystyle \int_{\Omega} \frac{u^2}{d_{\mathcal{E}}^2} \,dx.
\end{equation}
By \eqref{crescita} we easily  get 
\begin{equation} \label{a<b}
\displaystyle \int_{\Omega} F(\nabla  u)^2 \,dx \ge \alpha^2\int_{\Omega} |\nabla  u|^2 \,dx \ge 
\frac{\alpha^2}{4}\displaystyle \int_{\Omega} \frac{u^2}{d_{\mathcal{E}}^2} \,dx \ge \frac{1}{4}\frac{\alpha_{1}^2}{\alpha_{2}^2}\displaystyle \int_{\Omega} \frac{u^2}{d_F^2},
\end{equation}
where the constant in the right-hand side is smaller than $\frac 14$ since $\alpha_{1}<\alpha_{2}$. 
We emphasize that if $\Omega $ is not convex inequality  \eqref{a<b} holds under the assumption that $d_{\mathcal{E}}$ is $\mathcal{E}$-superharmonic, since \eqref{hardyeu} is in force. On the other hand  the assumption on $d_{\mathcal{E}}$ is not related with the hypotesis  required about  $d_F$ in the  Theorem \ref{hardyclass}, as observed in Remark \ref{equiv}.
\end{rem}

Using Theorem \ref{hardyclass} it is not difficult  to obtain a lower bound for the first eigenvalue of $\Delta_F $ defined in \eqref{lapani}.
\begin{cor}
Let $\Omega$ be a bounded domain of $\R^n$ and suppose that condition \eqref{C} holds. Let $\lambda_1(\Omega)$ be the first Dirichlet eigenvalue of the Finsler Laplacian, that is
\begin{equation}
  \label{eig}
  \lambda_1(\Omega) = \min_{\substack{u\in H^1_0(\Omega)\\ u\neq 0}}
\dfrac{\displaystyle \int_\Omega [F(\nabla u)]^{2}dx} {\displaystyle\int_\Omega |u|^2 dx}.
\end{equation}
Then 
\[
\lambda_1(\Omega) \ge \frac{1}{4r^2_F},
\]
where $r_F$ is the anisotropic inradius of $\Omega$ defined in \eqref{inrad}.
\end{cor}
\begin{proof}
Let $v$ the first eigenfunction related to $\lambda_1(\Omega)$ such that $\|v\|_{L^2(\Omega)}=1$. Then  \eqref{eig} and inequality \eqref{hardy} imply 
\begin{equation*}
  \lambda_1(\Omega) =\displaystyle \int_\Omega [F(\nabla v)]^{2}dx \ge \frac{1}{4} \int_{\Omega} \frac{v^2}{d_{F}^2} \,dx \ge \frac{1}{4r^2_F},
\end{equation*}
that is the claim.
\end{proof}

\section{Hardy inequality with a reminder term}
\begin{theo}
\label{reminder}
Let $\Omega$ be a domain of $\R^n$. Let us suppose also that condition \eqref{C} holds, and $\sup\{d_{F}(x,\de \Omega),x\in\Omega\}<+\infty$. Then for every function $u \in H_0^1(\Omega)$ the following inequality holds:
\begin{equation}
 \label{hardyrem}
\int_{\Omega} F^{2}(\nabla  u) \,dx - \frac{1}{4}\displaystyle \int_{\Omega} \frac{u^2}{d_F^2} \,dx \ge\frac14\int_{\Omega} \frac{u^{2}}{d_{F}^{2}}\left(\log\frac{D}{d_F}\right)^{-2}
dx,
 \end{equation}
where $D= e\cdot\sup\{d_{F}(x,\de \Omega),x\in\Omega\}$.
\end{theo}
\begin{proof}
We will use the following notation: 
\[
X(t)=-\frac{1}{\log t}, \quad t\in]0,1[.
\]
Being $F^{2}$ convex, we have that
\[
F(\xi_{1})^{2}\ge F(\xi_{2})^{2}+2F(\xi_{2})F_{\xi}(\xi_{2})\cdot (\xi_{1}-\xi_{2}).
\]
Let us consider
\[
\xi_{1}= \nabla u, \quad \xi_{2}=\frac u2 \frac{\nabla d_{F}}{d_{F}}\left[1-X\left(\frac{d_{F}}{D}\right)\right].
\] 
Being $d_{F}(x)\le \frac{D}{e}$, by the 1-homogeneity of $F$ we get
\begin{multline}
 F(\nabla u)^2\,dx \ge  \\
 \shoveright{\frac14\frac{u^{2}}{d^{2}}F^{2}(\nabla d)\left[1-X\left(\frac{d_{F}}{D}\right)\right]^{2}
 + \frac{u}{d_{F}}\left[1-X\left(\frac{d_{F}}{D}\right)\right]\times}
 \\[.2cm]
\shoveright{\times F(\nabla d)F_{\xi}(\nabla d)\cdot \left(\nabla u-\frac{u}{2}\frac{\nabla d_{F}}{d_{F}}\left[1-X\left(\frac{d_{F}}{D}\right)\right] \right)=}
\\[.2cm]
=-\frac{1}{4}\frac{u^2}{d_F^2}\left[1-X\left(\frac{d_{F}}{D}\right)\right]^{2}
+  \\+
\frac{u}{d_F}\left[1-X\left(\frac{d_{ F}}{D}\right)\right]  F_{\xi}(\nabla d_{F}) \cdot \nabla u
\label{eqeq1} 
\end{multline}
where last equality follows by $F(\nabla d_{F})=1$, the 1-homogeneity of $F$ and property \eqref{eq:eul}. 
Let us observe that, using the Divergence Theorem (in a general setting, contained for example in \cite{anz}), we have
\begin{multline}
\label{eqeq}
\int_{\Omega} \frac{u}{d_F}\left[1-X\left(\frac{d_{F}}{D}\right)\right]  F_{\xi}(\nabla d_{F}) \cdot \nabla u\, dx =\\=
-\int_{\Omega} \frac{u^{2}}{2} \divergenza\left( \left[1-X\left(\frac{d_{F}}{D}\right)\right]\frac{F_{\xi}(\nabla d_{F})}{d_{F}}  \right)dx=
\\[.2cm]= 
\int_{\Omega}\frac{u^{2}}{2} \left\{
\left[1-X\left(\frac{d_{F}}{D}\right)+X^{2}\left(\frac{d_{F}}{D}\right)\right]\frac{F_{\xi}(\nabla d_{F})\cdot \nabla d_{F}}{d_{F}^{2}}+\right.\\ 
\qquad\qquad\qquad\qquad\qquad\qquad\left.-\left[1-X\left(\frac{d_{F}}{D}\right)\right]\frac{\Delta_{F} d_{F}}{d_{F}}
\right\}dx
\ge\\[.2cm]
\ge \int_{\Omega} \frac{1}{2} \frac{u^{2}}{d_{F}^{2}}
\left[1-X\left(\frac{d_{F}}{D}\right)+X^{2}\left(\frac{d_{F}}{D}\right)\right]dx,
\end{multline}
where last inequality follows using the condition $-\Delta_{F} d_F \ge 0$.

Integrating \eqref{eqeq1}, and using \eqref{eqeq} we easily get

\begin{multline*}
 \displaystyle\int_{\Omega} F(\nabla u)^2\,dx \ge \int_{\Omega} \frac 1 4\frac{u^{2}}{d_{F}^{2}}\times
 \\
 \times \left\{
 -\left[1-X\left(\frac{d_{F}}{D}\right)\right]^{2}
+
2-2\,X\left(\frac{d_{F}}{D}\right)+2\,X^{2}\left(\frac{d_{F}}{D}\right)
 \right\}dx =\\
 =\frac{1}{4}\displaystyle \int_{\Omega} \frac{u^2}{d_F^2} \,dx +\frac14\int_{\Omega} \frac{u^{2}}{d_{F}^{2}}X^{2}\left(\frac{d_{F}}{D}\right)dx,
\end{multline*}
and the proof is completed.
\end{proof}

\begin{rem}
We observe that if $\Omega$ is a convex domain in $\R^{n}$, arguing  as in Remark \ref{alphabeta}, an  inequality of the type \eqref{hardyrem}  can be immediately obtained by using the following improved Hardy inequality involving $d_{\mathcal{E}}$ contained in \cite{bft}
\begin{equation}
	\int_{\Omega}|\nabla u|^{2} dx - \frac{1}{4}
	\int_{\Omega}\frac{|u|^{2}}{d_{{\mathcal{E}}}^{2}}dx \ge  \frac{1}{4} 
	 \int_{\Omega} \frac{ |u|^{2} }{ d_{\mathcal{E}}^{2}}
\left(\log\frac{D_{0}}{d_{\mathcal{E}}}\right)^{-2}	 
	 dx,
\end{equation}
where  $D_{0}\ge e\cdot\sup d_{\mathcal{E}}(x,\de \Omega)$  and  $u\in H_0^1(\Omega)$. Obviously also in this case it is not possible to obtain the optimal constants.
\end{rem}
\begin{cor}
Under the same assumptions of Theorem \ref{reminder}, the following anisotropic improved Hardy inequality holds
\begin{equation}
\label{h}
 \int_{\Omega} F(\nabla u)^{2} dx -  \frac{1}{4} \int_{\Omega} \frac{ |u|^{2} }{ d_F^{2}} dx \ge \frac{1}{4r_F^2} \int_{\Omega}|u|^{2} \,dx,
\end{equation}
where $r_F$ is the anisotropic inradius defined in \eqref{inrad}.
\end{cor}
\begin{proof}
By Theorem \ref{reminder}, to prove \eqref{h} it is sufficient  to show that 
\begin{equation*}
\label{app}
 \int_{\Omega} \frac{ |u|^{2} }{ d_F^{2}}
 \left(\log\frac{D}{d_F}\right)^{-2} dx \ge \frac{1}{r_F^2} \int_{\Omega}|u|^{2} \,dx.
\end{equation*}
This  is a consequence of the monotonicity of the following function
\[
f(t)= -t \log\left(\frac{t}{\text{e}\cdot r_F}\right), \quad 0<t\le r_F.
\]
Indeed  $f $ is strictly increasing and its maximum is $r_F$. This concludes the proof.
\end{proof}
An immediate consequence of the previous result is contained in the following remark.
\begin{rem}
Let $\Omega \subset \R^n$ be a bounded convex domain. Then there exists a positive constant $C(n)>0$ such that for any $u \in H_0^1(\Omega)$ we have  
 \begin{equation*}
\label{h}
 \int_{\Omega} F(\nabla u)^{2} dx -  \frac{1}{4} \int_{\Omega} \frac{ |u|^{2} }{ d_F^{2}}dx \ge C(n)|\Omega|^{-\frac{2}{n}} \int_{\Omega}|u|^{2} \,dx.
\end{equation*}
\end{rem}

\section{Optimality of the constants}
Here we prove the optimality of the constants and of the exponent which appear in the Hardy inequality \eqref{hardyrem}. More precisely, we prove the following result:
\begin{theo}
\label{opt}
Let $\Omega$ be a piecewise $C^{2}$ domain of $\R^{n}$. Suppose that the following Hardy inequality holds:
\begin{multline}
\label{ineqopt}
\int_{\Omega} F^{2}(\nabla  u) \,dx - A \displaystyle \int_{\Omega} \frac{u^2}{d_F^2} \,dx \ge B
\int_{\Omega} \frac{u^{2}}{d_{F}^{2}}
\left(\log\frac{D}{d_F}\right)^{-\gamma}
dx, \\  \forall u\in H_{0}^{1}(\Omega),
\end{multline}
for some constants $A>0$, $B\ge 0$, $\gamma>0$,
where $D= e\cdot\sup\{d_{F}(x,\de \Omega),x\in\Omega\}$.
Then:
\begin{enumerate}
\item[$(T_{1})$] \label{tesi1} $A\le \frac{1}{4}$;
\item[$(T_{2})$] \label{tesi2} If $A=\frac{1}{4}$ and $B>0$, then $\gamma\ge2$;
\item[$(T_{3})$] \label{tesi3} If $A=\frac14$ and $\gamma=2$, then $B\le \frac{1}{4}$.
\end{enumerate}
\end{theo}

\begin{proof}
The proof is similar to the one obtained in the Euclidean case by \cite{bft}. 
For the sake of completeness, we describe it in details. As before, 
let us denote by
\[
X(t)=-\frac{1}{\log t}, \quad t\in]0,1[.
\]
In order to prove the results we will provide a local analysis. 
Hence, we fix a Wulff shape $\mathcal W_{\delta}$ of radius $\delta$ centered 
at a point $x_{0}\in\de\Omega$. Being $\de\Omega$ piecewise smooth,
we may suppose that  for a sufficiently small $\delta$, 
$\de\Omega\cap \mathcal W_{\delta}$ is $C^{2}$.
Let now $\varphi$ be a nonnegative cut-off function in $C^{\infty}_{0}(\mathcal W_{\delta}\cap \Omega)$ such that $\varphi(x)=1$ for $x\in \mathcal W_{\delta/2}$. 
First of all, we prove some technical estimates which will be useful in the following. For $\eps>0$ and $\beta\in\R$ let us consider
\begin{equation}
\label{J}
J_{\beta}(\eps)=\int_{\Omega} \varphi^{2}d_{F}^{-1+2\eps}X^{-\beta}(d_{F}/D)\,dx.
\end{equation}
We split the proof in several claims.
\begin{claim}
The following estimates hold:
\begin{enumerate}
\item[(i)] $c_{1} \eps^{-1-\beta}\le J_{\beta}(\eps)\le c_{2} \eps^{-1-\beta}$, for $\beta>-1$, where $c_{1},c_{2}$ are positive constants independent of $\eps$;
\item[(ii)] $J_{\beta}(\eps)=\dfrac{2\eps}{\beta+1}J_{\beta+1}(\eps)+O_{\eps}(1)$, for $\beta >-1$;
\item[(iii)] $J_\beta(\eps)=O_{\eps}(1)$, for $\beta<-1$.
\end{enumerate}
\end{claim}
\textit{Proof of Claim 1}. By the coarea formula,
\begin{equation*}
J_{\beta}(\eps)= \int_{0}^{\delta} r^{-1+2\eps}X^{-\beta}(r/D) \left( \int_{d_{F}=r}\frac{\varphi^{2}}{|\nabla d_{F}|}dH^{n-1}\right)dr
\end{equation*}
Being $F(\nabla d_{F})=1$, by \eqref{crescita}, $0<\alpha_{1}\le |\nabla d_{F}|\le\alpha_{2}$ and
\[
0<C_{1}\le \int_{d_{F}=r}\frac{\varphi^{2}}{|\nabla d_{F}|}dH^{n-1}\le C_{2}.
\]
Then if $\beta<-1$, (iii) easily follows. Moreover, if $\beta >-1$, performing the change of variables $r=Ds^{1/\eps}$, (i) holds. 

As regards (ii), let us observe that
\[
\frac{d}{dr}X^{\beta}= 
\beta \frac{X^{\beta+1}}{r}.
\]
Recalling that $1=F(\nabla d_{F})F_{\xi}(\nabla d_{F})\cdot \nabla d_{F}$, then
\begin{align*}
\hspace{-.8cm}(\beta+1)J_{\beta}(\eps)=&
-\int_{\Omega}\varphi^{2} d_{F}^{2\eps} F(\nabla d_{F})F_{\xi}(\nabla d_{F})\cdot \nabla [X^{-\beta-1}(d_{F}/D)]dx=\\
=&\phantom{-}\int_{\Omega} \divergenza\left(
\varphi^{2} d^{2\eps}F(\nabla d_{F})F_{\xi}(\nabla d_{F})\right)X^{-\beta-1}(d_{F}/D)\,dx=\\
=& 2\int_{\Omega} \varphi d_{F}^{2\eps} X^{-\beta-1}(d_{F}/D)F(\nabla d_{F})F_{\xi}(\nabla d_{F})\cdot\nabla\varphi dx+\\
& \quad+2\eps\int_{\Omega}\varphi d_{F}^{2\eps-1}X^{-\beta-1}(d_{F}/D)\,dx
+\\
&\quad+\int_{\Omega} \varphi^{2} d_{F}^{2\eps} X^{-\beta-1}(d_{F}/D) \Delta_{F}d_{F}\,dx=\\
&= O_{\eps}(1) + 2\eps J_{\beta+1}(\eps). 
\end{align*}
We explicitly observe that 
\[
\int_{\Omega} \varphi^{2} d_{F}^{2\eps} X^{-\beta-1}(d_{F}/D) \Delta_{F}d_{F}\,dx =O_{\eps}(1)
\]
being $d_{F}$ a $C^{2}$ function in a neighborhood of the boundary (see \cite{cm07}). Then (ii) holds.

In the next claim we estimate the left-hand side of \eqref{ineqopt} when $u=U_{\eps}$, with
\[
U_{\eps}(x)=\varphi(x) w_{\eps}(x),\quad w_{\eps}(x)=d_{F}^{\frac12 +\eps} X^{-\theta}(d_{F}(x)/D),\quad \frac12<\theta<1.
\]
Let us define 
\[
\mathcal Q[U_{\eps}]:=\int_{\Omega} \left(F(\nabla U_{\eps} )^{2} - \frac{1}{4} \frac{U_{\eps}^{2}}{d_{F}^{2}} \right)dx.
\]
\begin{claim} The following estimates hold:
\begin{align}
\label{primo}
&\mathcal Q[U_{\eps}] \le \frac{\theta}{2} J_{2\theta-2} (\eps) +O_{\eps}(1), && \text{as }\eps\to 0;\\
\label{secondo}
&\int_{W_{\delta } \cap \Omega } F( \nabla U_{\eps})^{2}dx \le \frac{1}{4} J_{2\theta} (\eps) + O_{\eps}\left(\eps^{1-2\theta}\right), && \text{as }\eps\to 0.
\end{align}
\end{claim}
\textit{Proof of Claim 2}. The convexity of $F$ implies that
\[
F(\xi + \eta)^{2} \le F(\xi)^{2}+2 F(\xi)F(\eta)+F(\eta)^{2}, \quad \forall \xi,\eta\in \R^{n}.
\]
Hence by the homogeneity of $F$,
\begin{multline*}
\int_{\Omega} F(\nabla U_{\eps})^{2} dx \le   \\ \le 
\int_{\mathcal W_{\delta}\cap \Omega}   \varphi^{2} F^{2}(\nabla w_{\eps})dx+
\int_{\mathcal W_{\delta}\cap \Omega} w^{2}_{\eps} F^{2}(\nabla \varphi)dx + \\
\qquad\qquad\qquad\qquad\qquad+\int_{\mathcal W_{\delta}\cap \Omega} 2\varphi w_{\eps} F(\nabla\varphi)F(\nabla w_{\eps})
dx =
 \\
= \int_{W_{\delta}\cap \Omega} \varphi^{2} d_{F}^{2\eps-1}X^{-2\theta}(d_{F}/D)
\left(
\eps+\frac12 -\theta X(d_{F}/D)
\right)^{2} dx+ I_{1}+I_{2}.
\end{multline*}
As matter of fact,
\[
I_{2}\le C
 \int_{W_{\delta}\cap \Omega} d_{F}^{2\eps}X^{-2\theta}(d_{F}/D) dx =O_{\eps}(1);
\]
similarly, also $I_{1}=O_{\eps}(1)$. Then 

\begin{multline}
\label{5.12}
\mathcal Q[U_{\eps}]  \le \\
\le \int_{W_{\delta}\cap \Omega} \varphi^{2} d_{F}^{2\eps-1}X^{-2\theta}(d_{F}/D)
\left[\left(
\eps+\frac12 -\theta X(d_{F}/D)
\right)^{2}-\frac 14\right] dx\\+O_{\eps}(1) =
\\
\le \int_{W_{\delta}\cap \Omega} \varphi^{2} d_{F}^{2\eps-1}X^{-2\theta}(d_{F}/D)
\left(
\eps-\theta X(d_{F}/D)
\right)^{2} dx
+ \\ +
\int_{W_{\delta}\cap \Omega} \varphi^{2} d_{F}^{2\eps-1}X^{-2\theta}(d_{F}/D)
\left(
\eps-\theta X(d_{F}/D)
\right)
+O_{\eps}(1) =\\= a_{1}+a_{2}+O_{\eps}(1).
\end{multline}
Using (ii) of Claim 1 with $\beta=-1+2\theta$  we get 
\begin{equation}
\label{5.14}
a_{2}=O_{\eps}(1)  \quad \eps \to 0. 
\end{equation}
As regards $a_{1}$, similarly, applying  (ii) of Claim 1 with $\beta =2\theta-1$  in the first time and  $\beta =2 \theta -2$ in the second time we obtain 
\begin{equation}
\label{5.15}
a_{1}= \frac{\theta}{2}\int_{\mathcal{W}_{\delta}}\varphi^{2}d_{F}^{2 \eps-1}X^{2-2\theta}(d_{F}/D)\,dx + O_{\eps}(1).
\end{equation}
Then \eqref{primo} follows by \eqref{5.12},\eqref{5.14},\eqref{5.15} and \eqref{J}.
Finally observing that
\begin{equation}
\int_{\mathcal{W}_{\delta}\cap \Omega} F(\nabla U_{\eps})^{2}\,dx=\mathcal{Q}[U_{\eps}] +\frac 14 J_{2 \theta}(\eps),
\end{equation}
then the inequality \eqref{secondo} follows from \eqref{primo} and (ii) of Claim 1.

Now we are in position to conclude the proof of the Theorem.

Since inequality \eqref{ineqopt} holds for any $u \in H_{0}^{1}(\Omega)$ we take as test function $U_{\eps}$. Then by \eqref{secondo} and (i) of Claim 1 we have
\begin{equation*}
A \le \frac{1}{J_{2\theta}(\eps)} {\displaystyle\int_{\mathcal{W}_{\delta}\cap \Omega} 
F(\nabla U_{\eps})^{2}\, dx}
\le \frac14 +O_{\eps}(\eps).
\end{equation*}
Letting $\eps \to 0$  we obtain $(T_{1})$.

In order to prove $(T_{2})$ we put $A=\frac{1}{4}$ and reasoning by contradiction we assume that $\gamma <2$. As before by \eqref{primo} and (i) of Claim 1, we have
\begin{equation*}
0<B \le \displaystyle \frac{Q[U_{\eps}]}{J_{2\theta-\gamma}(\eps)} \le \displaystyle C\frac{\eps^{1-2\theta}}{\eps^{\gamma-1-2\theta}}= C \eps^{2-\gamma} \to 0 \, \text{as} \eps \to 0,
\end{equation*}
which is a contradiction and then $\gamma \ge 2$.

To conclude the proof of the Theorem we just  have to prove $(T_3)$.

If $A=\frac 14 $ and $\gamma=2$ then by \eqref{primo} we have
\begin{equation*}
B \le  \displaystyle \frac{Q[U_{\eps}]}{J_{2\theta-2}(\eps)} \le  \displaystyle \frac{\frac{\theta}{2}J_{2\theta-2}(\eps)+O_{\eps}(1)}{J_{2\theta-2}(\eps)}.
\end{equation*}
Then by assumption on $\theta $ and (i) of Claim 1 , letting $\eps \to 0$ we get
$$
B \le \frac{\theta}{2}.
$$  
Hence $(T_3)$ follows by letting $\theta \to \frac 12.$
\end{proof}
\begin{rem}
We stress that Theorem \ref{opt} assures that the involved constants in \eqref{hardyrem} are optimal, and also in the anisotropic Hardy inequality \eqref{hardy}. Actually, the presence of the remainder term in \eqref{hardyrem} guarantees that the constant $\frac14$ in \eqref{hardy} is not achieved.
\end{rem}
\textbf{Acknowledgement.}  This work has been partially supported by the FIRB 2013 project Geometrical and qualitative aspects of PDE's and by GNAMPA of INDAM.

\end{document}